\documentclass[reqno,11pt]{amsart} 
\usepackage{color}
\usepackage{mathtools}
\usepackage{amsfonts}
\usepackage{amssymb}
\usepackage{amsthm}
\usepackage{newlfont}
\usepackage{graphicx}
\usepackage{float}
\usepackage{amsmath}

\usepackage{fullpage}

\usepackage{tikz}
\usetikzlibrary{matrix,shapes,arrows,positioning,chains}

\numberwithin{equation}{section}
\newtheorem{thm}{Theorem}[section]
\newtheorem{lem}{Lemma}[section]
\newtheorem{rem}{Remark}[section]

\newtheorem{defn}{Definition}[section]

\newcommand{\ed}{\end {document}}

\newcommand{\wi}{\widehat}

\begin{document}

\title[F-modulated stability]{An F-modulated stability framework for multistep methods}

\author[D. Li]{ Dong Li}
\address{D. Li, SUSTech International Center for Mathematics, and Department of Mathematics, 
Southern University of Science and Technology, Shenzhen, China 518055, PRC}
\email{lid@sustech.edu.cn}

\author[C.Y. Quan]{Chaoyu Quan}	
\address{C.Y. Quan, SUSTech International Center for Mathematics, and Department of Mathematics, Southern University of Science and Technology,
	Shenzhen, P.R. China}
\email{quancy@sustech.edu.cn}

\author[W. Yang]{Wen Yang}
\address{\noindent W. Yang,~Wuhan Institute of Physics and Mathematics, Innovation Academy for Precision Measurement Science and Technology, Chinese Academy of Sciences, Wuhan 430071, P. R. China.}
\email{wyang@wipm.ac.cn}

\begin{abstract}
We introduce a new $\mathbf F$-modulated energy stability framework for general linear multistep
methods. We showcase the theory for the two dimensional molecular beam epitaxy model
with no slope selection which is a prototypical gradient flow with Lipschitz-bounded nonlinearity.
We employ a class of representative BDF$k$, $2\le k \le 5$ discretization schemes with explicit
$k^{\mathrm{th}}$-order extrapolation of the nonlinear term.  We prove the uniform-in-time boundedness of high Sobolev norms of the numerical
solution. The upper bound is unconditional, i.e. regardless of the size of the time step. We develop a new algebraic theory and calibrate nearly optimal and \emph{explicit}
maximal time step constraints which guarantee monotonic $\mathbf F$-modulated energy dissipation.
\end{abstract}

\maketitle

\section{Introduction}
Phase field models such as 
the Allen--Cahn (AC) equations \cite{allen1979microscopic}, the Cahn--Hilliard (CH) equations \cite{cahn1958free}, and the molecular beam epitaxy (MBE) models
\cite{clarke1987origin} have been widely used in material sciences, multiphase flow, biology, image processing and the like.  In this work for convenience of presentation we consider a prototypical
MBE model with no slope selection (MBE-NSS) posed on the two-dimensional periodic torus $\Omega =
\mathbb T^2 = [-\pi, \pi]^2$:
\begin{align} \label{1.1}
\partial_t u = - \nu \Delta^2 u - \nabla \cdot \Bigl( f ( \nabla u ) \Bigr),
\qquad (t,x) \in (0,\infty) \times \Omega,
\end{align}
where $f(z) = \frac {z}{1+|z|^2}$,  and $|z|=\sqrt{z_1^2+z_2^2}$ for $z =(z_1,z_2)^{\mathrm{T}}\in \mathbb R^2$. The real-valued function $u=u(t,x)$ is called
a scaled height function of the thin film in a co-moving frame.  The linear dissipative term  $-\nu \Delta^2 u$ represents to capillarity-driven isotropic surface diffusion (cf. Mullins \cite{m1957} and Herring \cite{h1951}) with $\nu>0$ being the diffusion coefficient.
 The dynamical evolution \eqref{1.1} can be derived from the $L^2$ gradient flow  of the energy functional
\begin{align} \label{1.3}
\mathcal E (u)= \int_{\Omega}
\Bigl( -\frac 12 \log(1+|\nabla u|^2) + \frac 12 \nu | \Delta u |^2 \Bigr) dx.
\end{align}
One should note that the sign in front of the logarithmic potential is negative which
is a manifestation of  the Ehrlich-Schwoebel effect.
Due to the strong competition between the potential term and the biharmonic diffusion
term,  an uphill atom current is often generated in the system which leads to
mound-like structures in the film. Under the assumption $|\nabla u |\ll 1$,  the energy
functional \eqref{1.3} can be roughly approximated by  a simpler-looking functional
\begin{align} \label{1.3a}
\mathcal E^{\mathrm{SS}} (u) = \int_{\Omega}
\Bigl( \frac 14 (|\nabla u|^2-1)^2 + \frac 12 \eta^2 |\Delta u|^2 \Bigr)dx.
\end{align}
The $L^2$-gradient flow of $\mathcal E^{\mathrm{SS}}$ corresponds to the  MBE model with
slope-selection, i.e. the system typically favors the slope $|\nabla u | \approx 1$ which
exhibits pyramidal structures. In stark contrast prototypical solutions to \eqref{1.1} have mound-like
structures whose slopes can have a large upper bound. On the other hand from the analysis point
of view, the system \eqref{1.1} is more benign than \eqref{1.3a} since the nonlinear term
has bounded derivatives of all orders. This is of fundamental importance since it leads to strong
a priori bound of the PDE solution uniformly in time. For smooth solutions to \eqref{1.1},  the mean-value of $u$ is preserved in time. For simplicity we set this mean value to be zero throughout
this work. The fundamental energy conservation law takes the form
\begin{align}
\mathcal E( u(t_2) ) + \int_{t_1}^{t_2} \| \partial_t u \|_2^2 dt = \mathcal
E (u(t_1) ), \qquad\forall\, 0\le t_1 <t_2 <\infty.
\end{align}
This yields
\begin{align} \label{1.6}
\mathcal E (u(t_2) ) \le \mathcal E (u(t_1) ), \qquad\forall\, t_2\ge t_1.
\end{align}
Since the mean value of $u$ is zero and the energy is coercive,  the estimate
\eqref{1.6} gives a priori global $H^2$ control of the solution. The wellposedness
and regularity of solutions to \eqref{1.1} follows  easily from this  and the fact that
the nonlinear term has bounded derivatives of all orders.

In the study of phase field models a fundamental problem is
 to design efficient, accurate and stable  numerical schemes which can accommodate vastly different spatial and temporal scales. A sampler of existing popular numerical methods includes  the convex-splitting scheme \cite{chen2012linear,eyre1998unconditionally,wang2010unconditionally}, the stabilization scheme \cite{shen2010numerical,xu2006stability}, the scalar auxiliary variable (SAV) scheme \cite{shen2018scalar}, semi-implicit/implicit-explicit (IMEX) schemes \cite{Li2021, LQT21, LTaam21} and so on.
Recently a new theoretical framework has
been established (\cite{Li2021, LQT21, LTaam21}) to analyze the stability and convergence 
of typical semi-implicit methods of order up to two.  However due to the lack of monotonic discrete
energy law there are very  few works in the literature devoted to the analysis of stability of higher order methods. In practical numerical implementations it is often observed that the energy of higher order methods typically
exhibits sporadic non-monotonic oscillations for medium time step sizes. As such it was already 
realized and heuristically argued in \cite{xu2006stability} that high order methods should dissipate
a suitably modified energy functional $\mathcal E_{\Delta t}$ which differs from the standard energy by
a minuscule $O(\Delta t^p)$ correction. In \cite{LYZ20}, by an ingenious cut-off procedure which caps 
the numerical solution with the PDE maximum principle, Li, Yang and Zhou proved rigorously
the $L^2$ stability of a class of high order methods for parabolic equations.  Concerning the
molecular beam epitaxy model with no slope selection,  Hao, Huang and Wang \cite{hao2020third} recently 
proposed a BDF3/AB3 scheme with a judiciously chosen additional stabilization term 
of the form $ - A \Delta t^2 \Delta^2 (u^{n+1}-u^n)$. It was rigorously shown in 
\cite{hao2020third} that if $A\ge O(\nu^{-1})$ then one can achieve unconditional energy
stability regardless of the size of the time step. In recent \cite{li2021bdf3},  a BDF3/EP3 
semi-discretization scheme was analyzed for the MBE model with no slope selection. 
Explicit and nearly optimal time step constraints were identified in \cite{li2021bdf3} for 
which the modified energy dissipation law is rigorously proved to hold. Moreover an unconditional
uniform energy bound was proved in \cite{li2021bdf3} with no size restrictions on the time step. 
However, whilst the analysis framework in \cite{li2021bdf3} is quite satisfactory for the BDFk methods of order less than three, it is by no means trivial to extend it to higher order methods such as
 BDFk, $k\ge 3$.  The purpose of this work is to develop further the program initiated in
 \cite{Li2021, LQT21, LTaam21, li2021bdf3} and construct a new $\mathbf F$-modulated energy
 stability framework for general linear multistep methods. 

To this end we first review the situation with general BDF$k$ schemes applied on \emph{linear 
models}. A general $q$-step method (cf. page 21 of \cite{Is}) for the ODE $y^{\prime}=f(y)$ takes the
form $
\sum_{j=0}^q a_j y^{n+j} = \tau \sum_{j=0}^q b_j f(y^{n+j}).
$
A method is of order $q\ge 1$ iff. 
$
\sum_{j=0}^q a_j y( (n+j) \tau) - \tau \sum_{j=0}^q b_j f( y((n+j)\tau) ) = \mathcal O(\tau^{q+1})
$
for all sufficiently smooth $y$ and the rate $\mathcal O(\tau^{q+1})$ cannot be improved
for some specific $y$. In terms of the polynomials $\rho
(w) = \sum_{j=0}^q a_j w^j$, $\sigma (w) = \sum_{j=0}^q b_j w^j$, this amounts to
requiring for some $c\ne 0$,
\begin{align}
\rho (w) - \sigma (w) \ln w = c (w-1)^{q+1} + \mathcal O( |w-1|^{q+2}).
\end{align}
A BDF$q$ method corresponds to specifying $\sigma(w) =(\sum_{j=1}^q\frac 1j) w^q$ and 
$\rho(w) = \sum_{j=1}^{q} \frac 1j \xi^{q-j} (\xi -1)^j$. In the literature one usually considers the
general Banach space ODE:
\begin{align}
u^{\prime} + A u =0, \quad 0<t<T,
\end{align}
where $A$ is a positive definite, self-adjoint, linear operator on a Hilbert space $H$ with dense
domain $D(A)$.  A $q$-step BDF method typically takes the form
\begin{align}
\sum_{j=0}^q a_j u^{n+j} + \tau A u^{n+q} =0, \qquad n=0, \cdots, N-q,
\end{align}
where the coefficients $a_j$ are extracted from the polynomial 
$\rho(w) =\sum_{j=1}^q \frac 1 j \xi^{q-j} (\xi-1)^j = \sum_{j=0}^q a_j \xi^j$. 
To make the energy method applicable to the parabolic equations, Nevanlinnna and
Odeh \cite{nevanlinna1981multiplier} introduced multipliers for BDF$k$ methods with $3\le k\le 5$. 
See also Lubich, Mansour and Venkataraman \cite{lubich2013backward} for a powerful application
in the stability analysis of parabolic equations. In recent \cite{akrivis2020energy}, Akrivis, Chen, Yu
and Zhou showed boundedness for the heat equation by using a novel multiplier for the BDF6
method. In the same work it was also shown that no Nevanlinna-Odeh multiplier exists. 
However, at present it is unknown whether these multiplier techniques can be used in the nonlinear
situations due to subtle technical obstructions.

In this work we propose a new theoretical framework for establishing the energy stability of extrapolated BDF$k$ schemes for general gradient flows with Lipschitz nonlinearities. 
To showcase our analysis we consider the model case MBE equations with no slope selection.
We shall not rely on any existing multiplier techniques, but will develop a completely new 
$\mathbf F$-modulated energy stability framework which guarantees the energy dissipation
with a very mild restriction on the time step (cf. Table \ref{tab0}).  One should note that
in Table \ref{tab0}, the explicit time step constraints are (as far as we know) the first of the
kind in the literature which accords very well with what is observed in practical numerical implementations.
\begin{table}[!]
\renewcommand\arraystretch{1.3}
\begin{center}
\def\temptablewidth{0.95\textwidth}
\caption{Optimal time step $\tau_{\max}$ for $\mathbf F$-modulated energy dissipation. }\vspace{-0.2in}\label{tab0}
{\rule{\temptablewidth}{1pt}}
\begin{tabular*}{\temptablewidth}{@{\extracolsep{\fill}}ccccc}
        & BDF2     & BDF$3$   &BDF$4$ & BDF$5$ \\ \hline 
$\tau_{\mathrm{max}}$   & $1.771626\nu$ & $0.211023\nu$  & $0.032644\nu$ & $0.001635\nu$
 \end{tabular*}
{\rule{\temptablewidth}{1pt}}
\end{center}
\end{table}
In the second part of our work, by another novel analysis  we show that the energy of BDF$k$ with
$2\le k\le 5$ remains uniformly bounded in $H^2$ which is also unconditional, i.e. regardless
of the size of the time step.

The rest of this paper is organized as follows.
In Section \ref{sect2}, we introduce the $\mathbf F$-modulated energy for the BDF$k$/EP$k$ scheme of the MBE-NSS equation and identify the explicit time step constraints for energy dissipation.
In Section \ref{sect3} we prove uniform energy boundedness of the numerical iterates with
no restrictions on the size of the time step.

\section{Energy dissipation of general BDF$k$ schemes}\label{sect2}

Classical BDF$k$ schemes for ODE $y' = f(y)$ takes the form (see \cite[pp. 173]{le07})
\begin{equation}\notag
\begin{aligned}
& \mathrm{BDF1: ~} y^{n+1} - y^n = \tau f( y^{n+1}) \quad (\mathrm{backward ~Euler}); \\
& \mathrm{BDF2: ~} \tfrac32 y^{n+2} - 2 y^{n+1} + \tfrac12 y^n =  \tau f(y^{n+2}); \\
& \mathrm{BDF3: ~} \tfrac{11}{6}y^{n+3} - 3 y^{n+2} + \tfrac32 y^{n+1} - \tfrac13 y^n =  \tau f(y^{n+3}); \\
& \mathrm{BDF4: ~} \tfrac{25}{12} y^{n+4} - 4 y^{n+3} + 3 y^{n+2} - \tfrac{4}{3} y^{n+1} + \tfrac{1}{4} y^n = \tau f(y^{n+4}); \\
& \mathrm{BDF5: ~}  \tfrac{137}{60} y^{n+5} - 5y^{n+4} + 5 y^{n+3} - \tfrac{10}{3} y^{n+2} + \tfrac{5}{4} y^{n+1} - \tfrac{1}{5} y^n = \tau f(y^{n+5}); \\
& \mathrm{BDF6: ~}  \tfrac{49}{20} y^{n+6} - 6 y^{n+5} + \tfrac{15}{2} y^{n+4} - \tfrac{20}{3} y^{n+3} + \tfrac{15}{4} y^{n+2} - \tfrac{6}{5} y^{n+1} + \tfrac{1}{6} y^n = \tau f(y^{n+6}).
\end{aligned}
\end{equation}
Methods with $k > 6$ are not zero-stable so they cannot be used.
For a fixed BDF$k$ method, the LHS of the above can be rewritten as
\begin{equation}\notag
\sum_{i=0}^k  A_i^{(k)} y^{n+k-i} = \sum_{i=0}^{k-1} a_i ^{(k)}(y^{n+k-i}-y^{n+k-i-1}) =: \sum_{i=0}^{k-1} a_i ^{(k)}\delta y^{n+k-i},
\end{equation}
where $a_0^{(k)} = A_0^{(k)}$, $a_i = \sum_{j=0}^i A_j^{(k)}$.

Consider the implicit-explicit extrapolated BDF$k$ scheme for the 2D MBE-NSS model:
\begin{equation}\label{eq:genBDFk}
\frac 1 \tau \sum_{i=0}^k A_i ^{(k)}u^{n+1-i} = -\nu \Delta^2 u^{n+1} + \nabla \cdot \left(f(\sum_{i=1}^k B_i^{(k)}\nabla u^{n+1-i})\right),
\end{equation}
where $\{A_i^{(k)}\}_{i=0}^k$ are coefficients of the $k$th-order backward differentiation formula (BDF$k$) and $\{B_i^{(k)}\}_{i=1}^k$ are the $k$th-order extrapolation (EP$k$) coefficients.
For simplicity, we rewrite \eqref{eq:genBDFk}  as
\begin{equation}\label{eq:genBDFk2}
\frac 1 \tau \sum_{i=0}^{k-1} a_i ^{(k)}\delta u^{n+1-i} = -\nu \Delta^2 u^{n+1} + \nabla \cdot \left(f(\nabla u^n + \sum_{i=1}^{k-1} b_i^{(k)}\nabla \delta u^{n+1-i})\right),
\end{equation}
where
\begin{equation}
\delta u^j = u^j -u^{j-1}.
\end{equation}
See Table \ref{tab1} for the specific values of $a_i^{(k)}$ and $b_i^{(k)}$.
\begin{table}[htb!]
\renewcommand\arraystretch{1.3}
\begin{center}
\def\temptablewidth{0.75\textwidth}
\caption{$a_i^{(k)}$ derived from BDF$k$ coefficients and $b_i^{(k)}$ from the EP$k$ coefficients, in \eqref{eq:genBDFk2}.}\vspace{-0.2in}
{\rule{\temptablewidth}{1pt}}\label{tab1}
\begin{tabular*}{\temptablewidth}{@{\extracolsep{\fill}}ccccccc}
   BDF$k$      &$a_{0}^{(k)}$     &$a_{1}^{(k)}$     &$a_{2}^{(k)}$
   &$a_{3}^{(k)}$  &$a_{4}^{(k)}$  & $a_5^{(k)}$\\  \hline
  $k=2$   &$\frac{3}{2}$   &$-\frac{1}{2}$   &&&\\[3pt]
  $k=3$   &$\frac{11}{6}$  &$-\frac{7}{6}$   &$\frac{1}{3}$   &&\\[3pt]
  $k=4$   &$\frac{25}{12}$ &$-\frac{23}{12}$ &$\frac{13}{12}$ &$-\frac{1}{4}$ &\\[3pt]
  $k=5$   &$\frac{137}{60}$&$-\frac{163}{60}$&$\frac{137}{60}$&$-\frac{21}{20}$ &$\frac{1}{5}$ \\[3pt]
  $k=6$   &$\frac{49}{20}$ &$-\frac{71}{20}$ &$\frac{79}{20}$ &$-\frac{163}{60}$ & $\frac{31}{30}$ & $-\frac{1}{6}$
\end{tabular*}
{\rule{\temptablewidth}{1pt}}
{\rule{\temptablewidth}{1pt}}
\begin{tabular*}{\temptablewidth}{@{\extracolsep{\fill}}ccccccc}
   EP$k$   &$b_{1}^{(k)}$  &$b_{2}^{(k)}$  &$b_{3}^{(k)}$  &$b_{4}^{(k)}$ & $b_{5}^{(k)}$ &\\  \hline
  $k=2$   &$ 1 $   &&&\\[3pt]
  $k=3$   &$2$   &$-1$   &&\\[3pt]
  $k=4$   &$3$ &$-3$ &$1$ &\\[3pt]
  $k=5$   &$4$&$-6$&$4$ &$-1$ \\[3pt]
  $k=6$   &$5$&$-10$&$10$ &$-5$ & $1$
\end{tabular*}
{\rule{\temptablewidth}{1pt}}
\end{center}
\end{table}

For the iterate $u^n$, the standard energy of MBE-NSS model is
\begin{equation}
\boxed{E^n = \mathcal E(u^n) =\int_{\Omega} \Bigl( \frac{\nu}{2}|\Delta u^n|^2 -\frac 12 
\log(1+|\nabla u^n|^2) \Bigr) \, d
 x}.
\end{equation}
It is not difficult to check that for $H^2$ solutions, the energy  is bounded from below.

\subsection{Energy dissipation}
We summarize below the property of the nonlinear function $f$. 
\begin{lem}[Property of $f$, see \cite{li2021bdf3} for the derivation]\label{lem1}
The following holds.
\begin{align}
|f(x)-f(y) | &\le  |x-y|, \qquad\forall\, x, y \in \mathbb R^2, \label{Lip0.1a}\\
x^T (Df)(z) x &\le L |x|^2, \qquad\forall\, x, z \in \mathbb R^2, \label{Lip0.1b}
\end{align}
with $L = \frac 18$. 
\end{lem}


\begin{defn}[$\mathbf F$-modulated energy]\label{def:energy}
Let $2\le k\le 6$.
Given some upper triangular matrix $\mathbf F\in \mathbb R^{(k-1)\times(k-1)}$, the $\mathbf F$-modulated energy for the BDF$k$ scheme \eqref{eq:genBDFk2}  is defined by
\begin{equation}\label{eq:modenergy}
\begin{aligned}
E^n_{\mathbf F,k} =~&E^n + \frac1\tau\left\langle  \mathbf F  v_n,v_n \right\rangle + \sum_{i = 1}^{k-1} c_i ^{(k)}\|\nabla \delta u^{n+1-i}\|^2,\\
=~&\frac{\nu}{2}\|\Delta u^n\|^2 -\int_\Omega\frac 12 \log(1+|\nabla u^n|^2) \, {\mathrm d} x+ \frac1\tau\left\langle  \mathbf F  v_n,v_n \right\rangle + \sum_{i = 1}^{k-1} c_i ^{(k)}\|\nabla \delta u^{n+1-i}\|^2,
\end{aligned}
\end{equation}
where 
$v_n= \left(\delta u^n,\delta u^{n-1} , \ldots,\delta u^{n-k+2}\right)^{\mathrm T}$, $\left\langle \cdot,\cdot \right\rangle$ denotes the $L^2$ inner product, and $c_i^{(k)} = \frac12\sum_{j=i}^{k-1} |b_j^{(k)}|$. We employ the following convention: for $\vec f=(f_1,\cdots,f_J)^{\mathrm T}$ and $\vec g=(g_1,\cdots,g_J)^{\mathrm T}$: $\Omega\to\mathbb{R}^J,$ 
\begin{align*}
\boxed{
\langle \vec f,\vec g\rangle=\sum_{j=1}^J\int_{\Omega}f_jg_jdx.}
\end{align*} 
For simplicity we define $c_k^{(k)} = 0$.
\end{defn}

Note that the definition of $\mathbf F$-modulated energy depends on the choice of $\mathbf F$.
We shall require $\mathbf F$ to be positive definite in order to preserve the positivity of $ E_{\mathbf F,k}^n$.
The following theorem rigorously establishes the energy dissipation under certain positivity
conditions.
\begin{thm}[Energy dissipation]\label{thm1}
Assume that there exist some upper triangular matrix $\mathbf F\in \mathbb R^{(k-1)\times(k-1)}$ and $\alpha>0$ such that
\begin{equation}\label{eq:cond}
\mathbf x^{\mathrm T} \mathbf U \mathbf x  \ge \alpha x_1^2, \quad\forall~\mathbf x =(x_1,x_2,\ldots,x_k)^{\mathrm T}\in \mathbb R^k,
\end{equation}
where
\begin{equation}\label{eq:defU}
\mathbf U \coloneqq \left(\begin{array}{cccc}a_0^{(k)} & a_1^{(k)} & \cdots & a_{k-1}^{(k)} \\0 & 0 & \cdots & 0 \\0 & 0 & \cdots & 0 \\0 & 0 & \cdots & 0\end{array}\right)
- \left(\begin{array}{cc} \mathbf F & \mathbf 0\\ \mathbf 0^{\mathrm T} & 0\end{array}\right)
 + \left(\begin{array}{cc} 0 & \mathbf 0^{\mathrm T} \\ \mathbf 0& \mathbf F\end{array}\right).
\end{equation}
Then the $\mathbf F$-modulated energy \eqref{eq:modenergy} of BDF$k$ scheme \eqref{eq:genBDFk2} decays w.r.t. $n$, i.e., $ E_{\mathbf F,k}^{n+1} \le  E_{\mathbf F,k}^n$,
under the mild restriction \begin{equation}\label{eq:restrict_tau}
0<\tau \le \frac{2\alpha\nu}{(\frac L {2}+2c_1^{(k)})^2} =:\beta \nu.
\end{equation}
Here $L=\frac 18$ is given in Lemma \ref{lem1} and $c_1^{(k)} =\frac12\sum_{j=1}^{k-1}|b_j^{(k)}|$.
\end{thm}
\begin{rem}
Specific $\alpha$, $\beta$, and $\mathbf F$ in Theorem \ref{thm1} are provided later in Table \ref{tab2}. 
\end{rem}
\begin{proof}
Multiplying \eqref{eq:genBDFk2} with $\delta u^{n+1}$ and integrating over $\Omega$, we have 
\begin{equation}\label{eq:2.11}
\left\langle\frac 1 \tau \sum_{i=0}^{k-1} a_i ^{(k)}\delta u^{n+1-i},\delta u^{n+1} \right\rangle = \frac1\tau\left\langle \mathbf A w_{n+1},w_{n+1}\right\rangle,
\end{equation}
where
\begin{equation}
\mathbf A = \left(\begin{array}{cccc}a_0^{(k)} & a_1^{(k)} & \cdots & a_{k-1}^{(k)} \\0 & 0 & \cdots & 0 \\0 & 0 & \cdots & 0 \\0 & 0 & \cdots & 0\end{array}\right),\qquad w_{n+1} = \left(\begin{array}{c}\delta u^{n+1} \\\delta u^n \\\vdots \\ \delta u^{n-k+2}\end{array}\right).
\end{equation}
On the RHS, the first term is
\begin{equation}\label{eq:thm_1}
\left\langle-\nu \Delta^2u^{n+1},\delta u^{n+1}\right\rangle  = \frac 12 \nu \|\Delta u^n\|^2 - \frac 12 \nu \|\Delta u^{n+1}\|^2 -\frac12 \nu \|\Delta \delta u^{n+1}\|^2.
\end{equation}
Denote 
\begin{equation}
F^n = -\int_\Omega\frac 12 \log(1+|\nabla u^n|^2) \, {\mathrm d} x.
\end{equation}
For the nonlinear term we have 
\begin{equation}\label{eq:thm_2}
\begin{aligned}
&\bigg\langle\nabla \cdot \Big(f(\nabla u^n + \sum_{i=1}^{k-1} b_i^{(k)}\nabla \delta u^{n+1-i})\Big), \delta u^{n+1}\bigg\rangle ~ \\
&  =  \left\langle\nabla \cdot (f(\nabla u^n)), \delta u^{n+1}\right\rangle - \bigg\langle f\Big(\nabla u^n + \sum_{i=1}^{k-1} b_i^{(k)}\nabla \delta u^{n+1-i}\Big)-f(\nabla u^n), \nabla \delta u^{n+1}\bigg\rangle\\
 &\stackrel{\text{\tiny by Lemma \ref{lem1}}}{\le}  \left\langle\nabla \cdot (f(\nabla u^n)), \delta u^{n+1}\right\rangle ~ +  \left\langle \sum_{i=1}^{k-1} \left|b_i^{(k)}\nabla \delta u^{n+1-i}\right|,\left|\nabla\delta u^{n+1}\right|\right\rangle ~\\
 &\stackrel{\text{\tiny by Lemma \ref{lem1}}}{\le}  F^n- F^{n+1}+ \frac L{2}\| \nabla (\delta u^{n+1})\|^2
+  \frac12 \sum_{i=1}^{k-1} |b_i^{(k)}|\left( \|\nabla \delta u^{n+1-i}\|^2+ \| \nabla (\delta u^{n+1})\|^2\right)\\
& = F^n- F^{n+1}+ \frac L{2}\| \nabla (\delta u^{n+1})\|^2
+  \left(\sum_{i=1}^{k-1} (c_i^{(k)}-c_{i+1}^{(k)}) \|\nabla \delta u^{n+1-i}\|^2\right)+ c_1^{(k)}\| \nabla (\delta u^{n+1})\|^2\\
& = F^n- F^{n+1}
+  \sum_{i = 1}^{k-1} c_i^{(k)} \|\nabla \delta u^{n+1-i}\|^2
- \sum_{i = 1}^{k-1} c_{i}^{(k)} \|\nabla \delta u^{n+2-i}\|^2+  \left(\frac L{2}+2c_1^{(k)}\right)\| \nabla (\delta u^{n+1})\|_2^2.
\end{aligned}
\end{equation}
By \eqref{eq:2.11}, \eqref{eq:thm_1} and \eqref{eq:thm_2}, we obtain 
\begin{equation}\label{2.16}
\begin{aligned}
&E^{n+1} -E^n +\sum_{i = 1}^{k-1} c_i ^{(k)}\|\nabla \delta u^{n+2-i}\|^2 - \sum_{i = 1}^{k-1} c_i ^{(k)}\|\nabla \delta u^{n+1-i}\|^2 
\\
&\le  -\frac1\tau\left\langle \mathbf A w_{n+1},w_{n+1}\right\rangle-\frac12 \nu \|\Delta \delta u^{n+1}\|^2 + \left(\frac L{2}+2c_1^{(k)}\right)\| \nabla (\delta u^{n+1})\|_2^2.
\end{aligned}
\end{equation}
Note that besides $E^{n+1}$, the LHS above involves pure quadratic (favorable) terms in $\delta u^{n+1},~\delta u^n,\ldots,\delta u^{n+3-k}$.
In order to harvest the coercivity and obtain strict energy dissipation, it turns out that we need to incorporate a further bilinear term
\begin{equation}
\frac1\tau\left\langle \mathbf F v_{n+1}, v_{n+1}\right\rangle = \frac1\tau\left\langle \mathbf F \left(\begin{array}{c}\delta u^{n+1} \\\delta u^n \\\vdots \\ \delta u^{n-k+3}\end{array}\right), \left(\begin{array}{c}\delta u^{n+1} \\\delta u^n \\\vdots \\ \delta u^{n-k+3}\end{array}\right)\right\rangle 
\end{equation}
into $E^{n+1}$ ($\langle \mathbf F v_n, v_n\rangle$ resp. for $E^n$). 
In yet other words the intricate pairwise interactions amongst $\delta u^{n+1},~\delta u^n,\ldots,\delta u^{n+3-k}$ has to be taken into account for energy dissipation.
To this end, we define the $\mathbf F$-modulated energy
\begin{equation}
E^n_{\mathbf F,k} = E^n + \frac1\tau\left\langle  \mathbf F  v_n,v_n \right\rangle + \sum_{i = 1}^{k-1} c_i ^{(k)}\|\nabla \delta u^{n+1-i}\|^2.
\end{equation}
In terms of $E^n_{\mathbf F,k}$, \eqref{2.16} takes the form (note that $w_{n+1}^{\mathrm T}=(v_{n+1}^{\mathrm T},\delta u^{n-k+2}) = (\delta u^{n+1},v_n^{\mathrm T})$, in yet other words, the first $k-1$ entries of $w_{n+1}$ is $v_{n+1}$ whereas the last $k-1$ corresponds to $v_n$)
\begin{equation}\label{2.16}
\begin{aligned}
E^{n+1}_{\mathbf F,k}-E^n_{\mathbf F,k}
&\le  -\frac1\tau\left\langle \mathbf A w_{n+1},w_{n+1}\right\rangle  + \frac1\tau\left\langle \mathbf F v_{n+1}, v_{n+1}\right\rangle-\frac1\tau\left\langle \mathbf F v_{n}, v_{n}\right\rangle
-\frac12 \nu \|\Delta \delta u^{n+1}\|^2 + \left(\frac L{2}+2c_1^{(k)}\right)\| \nabla (\delta u^{n+1})\|_2^2\\
& = -\frac1\tau\left\langle \mathbf A w_{n+1},w_{n+1}\right\rangle  + \frac1\tau\left\langle \left(\begin{array}{cc} \mathbf F & \mathbf 0\\ \mathbf 0^{\mathrm T} & 0\end{array}\right) w_{n+1}, w_{n+1}\right\rangle
-\frac1\tau\left\langle \left(\begin{array}{cc} 0 & \mathbf 0^{\mathrm T} \\ \mathbf 0& \mathbf F\end{array}\right) w_{n+1}, w_{n+1}\right\rangle\\
&\quad-\frac12 \nu \|\Delta \delta u^{n+1}\|^2 + \left(\frac L{2}+2c_1^{(k)}\right)\| \nabla (\delta u^{n+1})\|_2^2\\
& =-\frac1\tau\left\langle \mathbf U w_{n+1},w_{n+1}\right\rangle 
-\frac12 \nu \|\Delta \delta u^{n+1}\|^2 + \left(\frac L{2}+2c_1^{(k)}\right)\| \nabla (\delta u^{n+1})\|_2^2,
\end{aligned}
\end{equation}
where $\mathbf U$ is defined in \eqref{eq:defU}.
If the condition \eqref{eq:cond} is satisfied, we get
\begin{equation}
 E_{\mathbf F,k}^{n+1}- E_{\mathbf F,k}^n \le -\frac\alpha\tau \| \delta u^{n+1}\|^2 -\frac12 \nu \|\Delta \delta u^{n+1}\|^2 + \left(\frac L{2}+2c_1^{(k)}\right)\| \nabla (\delta u^{n+1})\|_2^2.
\end{equation}
Thus if
\begin{equation}
0<\tau \le\frac{2\alpha\nu}{(\frac L {2}+2c_1^{(k)})^2},
\end{equation}
then the energy dissipation property holds, i.e., $ E_{\mathbf F,k}^{n+1}- E_{\mathbf F,k}^n \le 0$.
\end{proof}

\subsection{Construction of $\mathbf F$}
By Theorem \ref{thm1}, it remains for us to find a suitable upper triangular matrix $\mathbf F$ fulfilling the condition \eqref{eq:cond}. To construct $\mathbf F$
it is of some importance to understand the structure of $\mathbf U$.
For example, if $k = 3$, then in terms of 
$\mathbf F=\left(\begin{array}{cc}f_{11} & f_{12} \\0 & f_{22}\end{array}\right)$, we have
\begin{equation}
\begin{aligned}
\mathbf U 
&= \left(\begin{array}{ccc}  a_0^{(3)} & a_1^{(3)} & a_{2}^{(3)} \\0 & 0 & 0 \\0 & 0 & 0\end{array}\right)
- \left(\begin{array}{ccc}  f_{11} & f_{12} & 0\\0 & f_{22} & 0 \\0 & 0 & 0\end{array}\right)
+\left(\begin{array}{ccc}  0 & 0&0 \\0 & f_{11} & f_{12} \\0 & 0 & f_{22}\end{array}\right) \\
& = \left(\begin{array}{ccc}  a_0^{(3)}-f_{11} & a_1^{(3)}-f_{12} & a_{2}^{(3)} \\0 & -f_{22}+f_{11} & f_{12} \\0 & 0 & f_{22}\end{array}\right).
\end{aligned}
\end{equation}
Observe that 
\begin{equation}
\mathbf U_{11} + \mathbf U_{22} +\mathbf U_{33} = a_0^{(3)}, \quad \mathbf U_{12} +\mathbf U_{23} = a_1^{(3)}, \quad
\mathbf U_{13} = a_2^{(3)}.
\end{equation}
More generally for $2\le k\le 6$, we have
\begin{equation}\label{eq:linsys}
\sum_{j-i = s} \mathbf U_{ij} = a_s^{(k)}, \quad s=0,\ldots,k-1.
\end{equation}
This condition turns out to be necessary and sufficient for the one-to-one correspondence of $\mathbf F$ and $\mathbf U$.
This is summarized as the following lemma. We omit the elementary proof. 
\begin{lem}[One-to-one correspondence of $\mathbf F$ and $\mathbf U$]
Given a matrix $\mathbf U$ satisfying the linear system \eqref{eq:linsys}, $\mathbf F$ can be determined uniquely by \eqref{eq:defU}.
\end{lem}
Somewhat surprisingly, the semi-positive definiteness of $\mathbf U$ readily leads to the semi-positive definiteness of $\mathbf F$.
Note that, however, this is only a sufficient condition in general. 
\begin{lem}[Semi-positive definiteness of $\mathbf F$]
\label{le2.3}
If $\mathbf U$ defined in \eqref{eq:defU} is semi-positive definite, then $\mathbf F$ is semi-positive definite.
\end{lem}
\begin{rem}
Here by semi-positive definiteness, we mean that 
\begin{equation}
\mathbf x^{\mathrm T} \mathbf U \mathbf x  \ge 0, \quad\forall~\mathbf x \in \mathbb R^k. 
\end{equation}
Clearly $\mathbf U$ is semi-positive definite $\Leftrightarrow$ $\frac12(\mathbf U + \mathbf U^{\mathrm T})$ is semi-positive definite. 
\end{rem}
\begin{proof}[Proof of Lemma \ref{le2.3}.]
We consider the case $k=3$. 
Note that 
\begin{equation}
\mathbf U= \left(\begin{array}{ccc}  a_0^{(3)}-f_{11} & a_1^{(3)}-f_{12} & a_{2}^{(3)} \\0 & -f_{22}+f_{11} & f_{12} \\0 & 0 & f_{22}\end{array}\right). 
\end{equation}
Since $\frac12(\mathbf U+\mathbf U^{\mathrm T})$ is semi-positive definite, we have 
\begin{equation}
f_{22}\ge 0,\quad  \mbox{det}\left|\begin{array}{cc} -f_{22}+f_{11}  & \frac{1}{2}f_{12} \\ \frac{1}{2}f_{12}& f_{22}\end{array}\right|\ge 0. 
\end{equation}
These imply that $\mathbf F$ is semi-positive definite. 
The case $k\ge 4$ follows along similar lines. 
\end{proof}

In the remainder of this section, we focus on 
\begin{align}
\notag
\boxed{\mbox{finding an upper triangular matrix $\mathbf U\in \mathbb R^{k\times k}$ satisfying both \eqref{eq:cond} and \eqref{eq:linsys}}.}
\end{align}
Note that the upper triangular matrix $\mathbf U$ has $\frac12 k(k+1)$ degrees of freedom, and we have to accommodate the inequality \eqref{eq:cond} together with $k$ equations \eqref{eq:linsys}.  In general, this is under-determined optimization problem. To simplify the analysis, we consider low rank  upper triangular $\mathbf U$ satisfying 
\begin{equation}\label{eq:Up}
\frac 1 2\left(\mathbf U + \mathbf U^{\mathrm T}\right) = \mathbf p\mathbf p^{\mathrm T}  + \alpha \mathbf e_1\mathbf e_1^{\mathrm T} 
\end{equation}
with prescribed $\mathbf p = (p_1,\ldots,p_k)^{\mathrm T}  $, $\alpha>0$ and $\mathbf e_1= (1,0,\ldots,0)^{\mathrm{T}}$. For given  $\mathbf p=(p_1,\cdots, p_k)^{\mathrm{T}}\in \mathbb{R}^k,~\alpha>0$, we define
\begin{equation}\label{eq:U_form}
\mathbf U_{ij} = 
\left\{
\begin{aligned}
& p_i^2+\alpha\delta_{i1}  && i=j,\\
& 2p_ip_j && i<j,\\
& 0 && i>j.
\end{aligned}
\right.
\end{equation}
Obviously, $\mathbf U+\mathbf U^{\mathrm T}$ has rank less than or equal to $2$ and $\mathbf U$ satisfies the positive definiteness property \eqref{eq:cond} with the same $\alpha$. The restriction \eqref{eq:linsys} imposes the following conditions on $\mathbf p\in\mathbb{R}^k$ and $\alpha>0$:
\begin{equation}\label{eq:nonlinsys}
\left\{
\begin{aligned}
& p_1^2 + \ldots+p_{k}^2 = a^{(k)}_0-\alpha,\quad (s=0);\\
& 2 \sum_{j=1}^{k-s} p_j p_{j+s}= a^{(k)}_s,\quad s= 1,\ldots,k-1.
\end{aligned}
\right.
\end{equation}
In yet other words, we have reduced the proof of energy dissipation to solving a  set of $k$ quadratic equations with $k+1$ unknowns! (The values of $a_s^{(k)}$ are specified in Table \ref{tab1}.)

For given $k=2,\cdots,6$, we define the following threshold:
\begin{equation}
\alpha^{(k)}_{\mathrm{max}} \coloneqq \sup_\alpha\left\{\alpha \in \mathbb R \mid \eqref{eq:nonlinsys} \mbox{ is solvable}\right\}.
\end{equation}
Summing all equations in \eqref{eq:nonlinsys} and using the fact $\sum\limits_{i=0}^{k-1}a_i^{(k)}=1$, we obtain
\begin{equation}\label{eq:lin1a}
\left(\sum_{i=1}^k p_i\right)^2 =1-\alpha.
\end{equation}
Similarly using alternating sum, we have
\begin{equation}\label{eq:lin2a}
\left(\sum_{i=1}^k (-1)^{i-1} p_i \right)^2= \sum_{i=0}^{k-1} (-1)^i a_i^{(k)}-\alpha.
\end{equation}
Note that \eqref{eq:lin1} implies $\alpha_{\mathrm{max}}^{(k)} \le 1$.
Theoretically speaking, it is best to take largest $\alpha$ in order to saturate the upper bound in \eqref{eq:restrict_tau}.

\begin{lem}[BDF2]
For the BDF2 scheme, $\alpha_{\mathrm{max}}^{(2)} = 1$ is reached when $\mathbf p= \left(\frac12,-\frac12\right)^{\mathrm T}$ in \eqref{eq:Up}.
\end{lem}
\begin{proof}
Direct computation.
\end{proof}

\begin{lem}[BDF3]
For the BDF3 scheme, $\alpha_{\mathrm{max}}^{(3)} = \frac{95}{96}$ is reached when $\mathbf p = \left(\frac1{\sqrt6},-\frac{7}{4\sqrt6},\frac1{\sqrt6}\right)^{\mathrm T}$ in \eqref{eq:Up}.
\end{lem}
\begin{proof}
We first rewrite \eqref{eq:lin1a} and \eqref{eq:lin2a} as
\begin{align}
&\sum_{i=1}^k p_i=\pm(1-\alpha)^{\frac12}, \label{eq:lin1}\\
&\sum_{i=1}^k (-1)^{i-1} p_i = \pm \left(\sum_{i=0}^{k-1} (-1)^i a_i^{(k)}-\alpha\right)^{\frac12}.\label{eq:lin2}
\end{align}
Note that if $\mathbf p$ is a solution to \eqref{eq:nonlinsys}, then $-\mathbf p$ is also a solution. 
This implies that we should consider two cases: the right-hand sides of \eqref{eq:lin1} and \eqref{eq:lin2} have the same sign or the opposite sign. 

From \eqref{eq:nonlinsys}, \eqref{eq:lin1}, and \eqref{eq:lin2}, we consider the case when the right-hand sides of \eqref{eq:lin1} and \eqref{eq:lin2} have same sign:
\begin{equation}
\left\{
\begin{aligned}
p_1+p_2+p_3 & = (1-\alpha)^{\frac12},\\
p_1-p_2+p_3 & = \left(\frac{10}{3}-\alpha\right)^{\frac12},\\
p_1p_3 & = \frac{1}{6},
\end{aligned}
\right.
\end{equation}
which yields
\begin{equation}
\begin{aligned}
& p_2  = \frac12\left((1-\alpha)^{\frac12}-\left(\frac{10}{3}-\alpha\right)^{\frac12}\right),\\
& p_1+p_3  = \frac12\left((1-\alpha)^{\frac12}+\left(\frac{10}{3}-\alpha\right)^{\frac12}\right),\quad
p_1p_3  = \frac16.
\end{aligned}
\end{equation}
Thus the above system is solvable if and only if $p_1p_3 \le \frac14(p_1+p_3)^2$, i.e.,
\begin{equation}\alpha \le \alpha_{\mathrm{max}}^{(3)} = \frac{95}{96}.\end{equation}
When $\alpha = \frac{95}{96}$, we can obtain $\mathbf p = \left(\frac1{\sqrt6},-\frac{7}{4\sqrt6},\frac1{\sqrt6}\right)^{\mathrm T}$.

In addition, it is not difficult to check that in the case when the right-hand sides of \eqref{eq:lin1} and \eqref{eq:lin2} have opposite signs, $\alpha$ can not reach $\frac{95}{96}$. 
We omit this computation here. 
\end{proof}


\begin{lem}[BDF4]
For the BDF4 scheme,  $\alpha_{\mathrm{max}}^{(4)}\approx 0.814139$, when
\begin{equation}
{\mathbf p}\approx (-0.223519,  0.719240,  -0.623843,   0.559237)^{\mathrm T}
\end{equation}
 in \eqref{eq:Up}.
\end{lem}

\begin{proof}
When $k = 4$, \eqref{eq:nonlinsys} can be written explicitly as
\begin{equation}\label{2.40}
\begin{cases}
p_1^2+p_2^2+p_3^2+p_4^2 = \frac{25}{12}-\alpha;\\
p_1p_2+p_2p_3+p_3p_4 =-\frac{23}{24};\\
p_1p_3+p_2p_4 = \frac{13}{24};\\
p_1p_4 =-\frac18.
\end{cases}
\end{equation}
Let $a=\dfrac{p_2}{p_1},~b=\dfrac{p_3}{p_2},~c=\dfrac{p_4}{p_3}$. 
Clearly
\begin{equation}
abc = \frac{p_4}{p_1}<0.
\end{equation}
On the other hand, if $a,b,c\in \mathbb R$ are given satisfying $abc<0$, then $\mathbf p$ are solvable:
\begin{equation}\label{2.42}
p_4 = \pm\left({-\frac18 abc}\right)^{\frac12}, ~p_3 = \frac{p_4}{c},~p_2 = \frac{p_3}{b},~p_1 = \frac{p_2}{a}.
\end{equation}
Substituting \eqref{2.42} with $p_4 = \left({-\frac18 abc}\right)^{\frac12}$  into \eqref{2.40}, we get
\begin{equation}
\label{eq:abc}
\begin{cases}
\alpha=\dfrac{25}{12}+\dfrac18\left(\dfrac{1}{abc}+\dfrac{a}{bc}+\dfrac{ab}{c}+abc\right),\\
\\	
\dfrac{1}{bc} + \dfrac{a}{c} + ab=\dfrac{23}{3},\\
\\
\dfrac{1}{c} + a=-\dfrac{13}{3}.
\end{cases}
\end{equation}
A simple computation yields
\begin{equation}
\begin{aligned}
b=\frac{23 + 13 a + 3 a^2 \pm (529 + 754 a + 343 a^2 + 78 a^3 + 9 a^4)^{\frac12}}{6 a}\quad \mathrm{and}\quad c=-\dfrac{3}{13+3a}.
\end{aligned}
\end{equation}
Here, $b\in \mathbb R$ if and only if 
$
529 + 754 a + 343 a^2 + 78 a^3 + 9 a^4 \ge 0,
$
i.e., 
\begin{equation}
a \le \frac16 \left(-13 - \sqrt{-179 + 24 \sqrt{78}}\right) =:a_{\mathrm L} \quad\mbox{or}\quad a \ge \frac16 \left(-13 + \sqrt{-179 + 24 \sqrt{78}}\right)=:a_{\mathrm R} .
\end{equation}
Moreover, the restriction $abc<0$ forces
\begin{equation}
\frac{23 + 13 a + 3 a^2 \pm (529 + 754 a + 343 a^2 + 78 a^3 + 9 a^4)^{\frac12}}{13+3a}>0.
\end{equation}
Collecting the estimates, we obtain two families of solutions:
\begin{equation}
\begin{aligned}
& a \in \left(-{13}/{3},a_{\mathrm L}\right]\cup \left[a_{\mathrm R},\infty\right),\\
& b=\frac{23 + 13 a + 3 a^2 + (529 + 754 a + 343 a^2 + 78 a^3 + 9 a^4)^{\frac12}}{6 a},
\end{aligned}
\end{equation}
or
\begin{equation}
\begin{aligned}
& a \in \left(-\infty,-13/3\right)\cup\left(-13/3,a_{\mathrm L}\right] \cup \left(a_{\mathrm R},0\right),\\
& b=\frac{23 + 13 a + 3 a^2 - (529 + 754 a + 343 a^2 + 78 a^3 + 9 a^4)^{\frac12}}{6 a}.
\end{aligned}
\end{equation}
The main task now is to find the maximal $\alpha=\alpha(a)$ in \eqref{eq:abc}. 
In Figure \ref{fig:alpha_a}, we plot $\alpha$ w.r.t. $a$ corresponding to the above two families of solutions. 
Rigorous numerical computation leads to the maximum value  $\alpha_{\mathrm{max}}^{(4)}\approx 0.814139$, when
\begin{equation*}
{\mathbf p}\approx (-0.223519,  0.719240,  -0.623843,   0.559237)^{\mathrm T}.
\end{equation*}
\begin{figure}
\includegraphics[trim={1.5in 0 1.5in 0},clip,width = 1\textwidth]{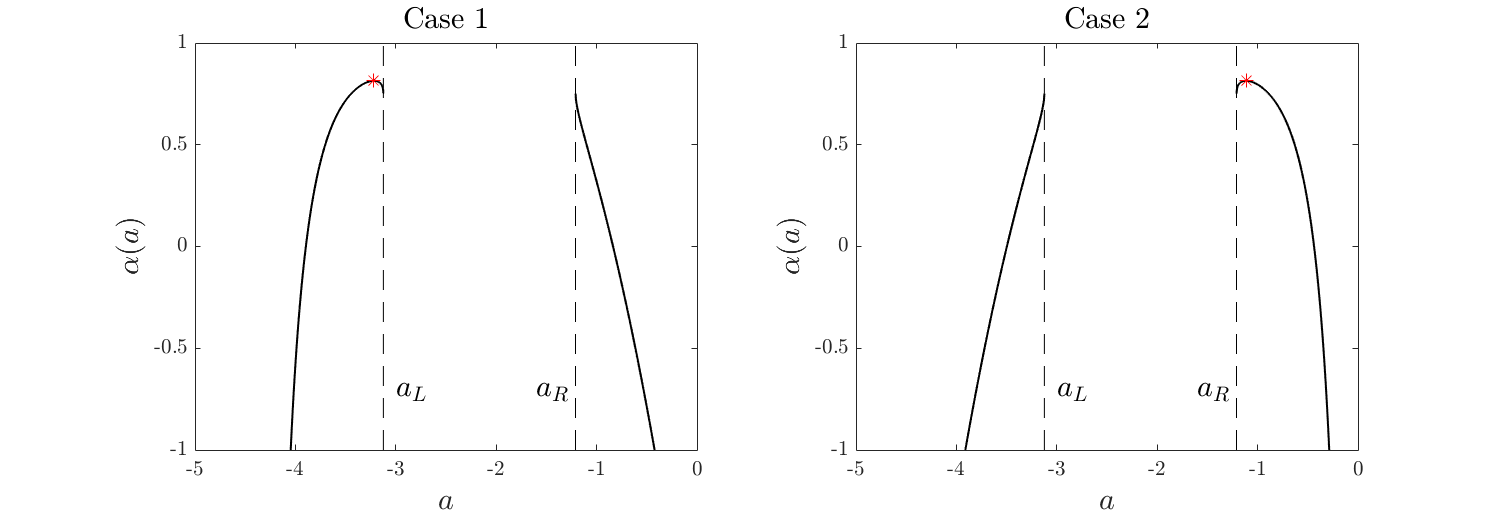}
\caption{$\alpha$ w.r.t. $a$ for BDF4 scheme. The red star marker denotes the maximum value of $\alpha$.}\label{fig:alpha_a}
\end{figure}
\end{proof}

\begin{lem}[BDF5]
For the BDF5 scheme,  $\alpha_{\mathrm {max}}^{(5)} \approx 0.185545$, when
\begin{equation}
\mathbf p \approx (0.868686,  -0.448459,   0.912060,  -0.544932, 0.115116)^{\mathrm T}
\end{equation}
 in \eqref{eq:Up}.
\end{lem}

\begin{proof}
When $k = 5$, \eqref{eq:nonlinsys} can be written explicitly as
\begin{equation}\label{2.50}
\begin{cases}
p_1^2+p_2^2+p_3^2+p_4^2 + p_5^2= \frac{137}{60}-\alpha;\\
p_1p_2+p_2p_3+p_3p_4 + p_4p_5=-\frac{163}{120};\\
p_1p_3+p_2p_4 + p_3p_5 = \frac{137}{120};\\
p_1p_4 + p_2p_5=-\frac{21}{40};\\
p_1p_5 = \frac1{10}.
\end{cases}
\end{equation}
Let $a=\dfrac{p_2}{p_1},~b=\dfrac{p_3}{p_2},~c=\dfrac{p_4}{p_3},~d = \dfrac{p_5}{p_4}$. 
Clearly
\begin{equation}
abcd = \frac{p_5}{p_1}>0.
\end{equation}
On the other hand, if $a,b,c,d\in \mathbb R$ are given satisfying $abcd>0$, then $\mathbf p$ are solvable:
\begin{equation}
p_5 = \pm\left({\frac1{10} abcd}\right)^{\frac12},~p_4 = \frac{p_5}{d}, ~p_3 = \frac{p_4}{c},~p_2 = \frac{p_3}{b},~p_1 = \frac{p_2}{a}.
\end{equation}
The simplified system becomes
\begin{equation}
\label{eq:abcd}
\begin{cases}
\alpha=\dfrac{137}{60}-\dfrac1{10}\left(\dfrac{1}{abcd}+\dfrac{a}{bcd}+\dfrac{ab}{cd}+\dfrac{abc}{d}+abcd\right),\\
\\	
\dfrac{1}{bcd} + \dfrac{a}{cd} + \dfrac{ab}{d} + abc=-\dfrac{163}{12},\\
\\
\dfrac{1}{cd} + \dfrac{a}{d}+ab=\dfrac{137}{12},\\
\\
\dfrac{1}{d} +a=-\dfrac{21}{4}.
\end{cases}
\end{equation}
Clearly for $a\ne -\frac{21}{4}$,
\begin{equation}\label{eq:cd}
\begin{aligned}
d = \left(-\dfrac{21}{4}-a\right)^{-1},\quad c = \frac{-\frac{21}{4}-a}{\frac{137}{12}+a^2+\frac{21}{4}a-ab},
\end{aligned}
\end{equation}
and $b$ satisfies a cubic equation:
\begin{equation}\label{eq:ab_cubic}
\begin{aligned}
&b\left(\frac{163}{12}-\left(\frac{21}{4}+a\right)ab\right) \left(\frac{137}{12}+a^2+\frac{21}{4}a-ab\right)\\
&+(1+ab)\left(\frac{137}{12}+a^2+\frac{21}{4}a-ab\right)^2 
-\left(\frac{21}{4}+a\right)ab^2 = 0.
\end{aligned}
\end{equation}
For fixed $a$, this cubic equation in $b$ has one or three real roots. 
Since
\begin{equation}
abcd = \frac{ab}{\frac{137}{12}+a^2+\frac{21}{4}a-ab}>0,
\end{equation}
we must impose (note that the case $ab<0$ is excluded since $\frac{137}{12}+a^2+\frac{21}{4}a>0$)
\begin{equation}\label{eq:ab_restrict}
\begin{aligned}
ab>0\quad\mathrm{and}\quad
\frac{137}{12}+a^2+\frac{21}{4}a-ab>0.
\end{aligned}
\end{equation}
Regarding $a$ as a parameter, $b$ is obtained by solving the cubic equation \eqref{eq:ab_cubic} together with the constraint \eqref{eq:ab_restrict}.
The other two variables $c$ and $d$ are computed via \eqref{eq:cd}.
The governing variable $\alpha(a)$ can be computed from the first equation in \eqref{eq:abcd}.
A rigorous numerical computation gives
$\alpha_{\mathrm {max}}^{(5)} \approx 0.185545$ with
\begin{equation}
\mathbf p \approx (0.868686,  -0.448459,   0.912060,  -0.544932, 0.115116)^{\mathrm T}.
\end{equation}
\end{proof}

\begin{rem}
Preliminary numerical experiments suggest that for BDF6, $\alpha_{\mathrm{max}}^{(6)}<0$. 
An interesting further issue is to determine the corresponding threshold for higher rank matrices. 
However we will not dwell on this subtle technicality here. 
\end{rem}

For readers' convenience we summarize the main results obtained in this section in Table \ref{tab2}.

\begin{table}[!]
\renewcommand\arraystretch{1.3}
\begin{center}
\def\temptablewidth{0.95\textwidth}
\caption{Optimal $\mathbf F$ in the energy \eqref{eq:modenergy} corresponding to $\alpha_{\mathrm{max}}^{(k)}$ and $\beta_{\mathrm{max}}^{(k)}$ in \eqref{eq:restrict_tau} for the BDF$k$/EP$k$ scheme \eqref{eq:genBDFk} of 2D MBE-NSS model. If the time step $\tau \le\beta_{\mathrm{max}}^{(k)}\nu$, the $\mathbf F$-modulated energy decays w.r.t. time. }\vspace{-0.2in}\label{tab2}
{\rule{\temptablewidth}{1pt}}
\begin{tabular*}{\temptablewidth}{@{\extracolsep{\fill}}cccc}
   BDF$k$      &$\mathbf F$     &$\alpha_{\mathrm{max}}^{(k)}$     &$\beta_{\mathrm{max}}^{(k)}$ \\ \hline 
 $k=2$   & $\dfrac14$   &$1$  & $\frac{512}{289}$ \\[10pt]
  $k=3$   &$\left(\begin{array}{rr} \frac{65}{96}  & -\frac{7}{12}  \\  & \frac16\end{array}\right)$  &$\frac{95}{96}$   &$\frac{1520 }{7203} $   \\[10pt]
  $k=4$   & $\left(\begin{array}{rrr} 1.219233& -1.595139&  0.804452 \\ & 0.701927& -0.697753 \\  &  & 0.312746\end{array}\right) $ &$0.814139$ &$0.032644$ \\[10pt]
  $k=5$   & $\left(\begin{array}{rrrr}1.343172& -1.937526& 0.698746&   -0.10325 \\ & 1.142056 & -1.119483 &   0.209986 \\ &  & 0.310203& -0.125461 \\ &  &  & 0.013251\end{array}\right)$ &$0.185545$ & $0.001635$
  \end{tabular*}
{\rule{\temptablewidth}{1pt}}
\end{center}
\end{table}


\section{Uniform boundedness of energy for any $\tau>0$}\label{sect3}
In this section, we consider the MBE-NSS equation defined in $\mathbb T^2 \coloneqq [0,1]^2$. 
Clearly, the average height is conserved in time, i.e.
\begin{equation}
\int_{\mathbb{T}^2}u(t,x) dx =  \int_{\mathbb{T}^2}u_0(x) dx, \qquad\forall\, t>0.
\end{equation}

From the energy dissipation analysis in Section \ref{sect2}, we have already established the $H^2$ bound of $u^n$ for the BDF$k$ scheme  when 
$
0< \tau  \le \tau_0 = \beta_{\mathrm{max}}^{(k)} \nu$, 
So we only consider the case of $\tau>\tau_0$ in what follows.

\begin{thm}[Uniform boundedness of energy for arbitrary time step]
\label{th3.1}
Consider the scheme \eqref{eq:genBDFk}. Assume $u^0,\cdots,u^{k-1}\in H^2(\mathbb T^2)$ satisfy (below recall $\delta u^i = u^i-u^{i-1}$)
\begin{itemize}
\item
$
\displaystyle\int_{\mathbb T^2}u^{k-1}dx=\cdots=\int_{\mathbb{T}^2}u^0dx,
$
\item
$
\displaystyle\sum_{i=1}^{k-1}\|\delta u^i\|_2^2 \le\alpha_k\tau,
$
where $\alpha_k>0$ is some constant. 
\end{itemize}
Then we have the following uniform $H^2$ bound on all numerical iterates: 
\begin{equation}
\sup_{n \ge 0}\left(\|u^n\|_2+\|\Delta u^n\|_2\right) \le B_k<+\infty,
\end{equation}	
where $B_k>0$ depends only on $(h^0,\cdots,h^{k-1},\varepsilon,\alpha_k)$. In particular $B_k$ does not depend on $\tau.$
\end{thm}

In order to prove Theorem \ref{th3.1}, we consider the following scheme (see the paragraph preceding \eqref{eq:genBDFk} for the definition of $A_i^{(k)}$):
\begin{equation}
\label{3.model}
\frac{1}{\tau}\sum\limits_{i=0}^kA_i^{(k)}u^{n+1-i}=-\Delta^2u^{n+1}+\nabla\cdot f^n,\quad n \ge k-1.
\end{equation}
Here, $f^n$ denotes some approximation of $f(u(t_{n+1}))$ such as the extrapolation term. 
We have the following uniform boundedness result.

\begin{thm}[$H^2$-bound]\label{thm:H2bound}
Consider the scheme \eqref{3.model} with $\tau \ge\tau_0>0$. Assume that $u^0,\cdots,u^{k-1}\in H^2(\mathbb T^2)$ and have mean zero. Suppose that for some $\gamma_0>0$,
\begin{equation}
\sup_{n \ge k}\|f^n\|_2 \le \gamma_0<+\infty.
\end{equation}	
We have 
\begin{equation}
\sup_{n \ge k}\left(\|u^n\|_2+\|\Delta u^n\|_2\right) \le \gamma_1<+\infty,
\end{equation}
where $\gamma_1>0$ depends only on $(\tau_0,\gamma_0,u^0,\cdots,u^{k-1}).$ 	
\end{thm}

\begin{proof}
We rewrite equation \eqref{3.model} as 
\begin{equation}
\label{3.linear}
u^{n+1}=-\sum_{i=1}^k\frac{A_i^{(k)}}{A_0^{(k)}+\tau\Delta^2}u^{n+1-i}+\tau\frac{1}{A_0^{(k)}+\tau\Delta^2}\nabla\cdot f^n.
\end{equation}
Since we are working with mean-zero functions,  \eqref{3.linear} can be recast as
\begin{equation}
\label{3.fourier}
u^{n+1}=-\sum_{i=1}^kA_i^{(k)}T u^{n+1-i}+\tau T\nabla\cdot f^n,
\end{equation}
where $T$ is a Fourier multiplier defined by
$$\widehat T(j)=
\frac{1}{A_0^{(k)}+\tau|j|^4}1_{|j| \ge1}.$$
It is not difficult to verify that
\begin{equation*}
0<\frac{1}{A_0^{(k)}+\tau|j|^4} \le\frac{1}{A_0^{(k)}+\tau_0},\quad
\tau|j|\frac{1}{A_0^{(k)}+\tau|j|^4} \le 1,\qquad \forall~ 0\neq j\in\mathbb{Z}^2.
\end{equation*}
Consequently,
\begin{equation}
\label{3.linear-j}
0<\widehat T(j) \le\frac{1}{A_0^{(k)}+\tau} \le\frac{1}{A_0^{(k)}+\tau_0},
\quad \tau|j||\widehat T(j)| \le 1,\qquad \forall ~0\neq j\in\mathbb{Z}^2.
\end{equation}

We set (below $\mathfrak{i}=\sqrt{-1}$)
\begin{equation}
\begin{aligned}
Z^{n+1}(j)=&\left(\widehat{u^{n+1}}(j),\widehat{u^{n}}(j)\cdots,\widehat{u^{n-k+1}}(j)\right)^{\mathrm T},\\
F^{n+1}(j)=&\left(\mathfrak{i}\tau \widehat T(j)j\wi{f^n}(j),0,\cdots,0\right)^{\mathrm T},
\end{aligned}
\end{equation}
and
\begin{equation}
M(j)=\left(
\begin{matrix}
-A_1^{(k)}\wi T(j)  &  -A_2^{(k)}\wi T(j) &\cdots &-A_{k-1}^{(k)}\wi T(j)&-A_{k}^{(k)}\wi T(j)\\
1  & 0 &\cdots &0&0\\
0  & 1 &\cdots &0&0\\
\vdots &\vdots &\ddots &\vdots&\vdots\\
0&0&\cdots&1&0
\end{matrix}
\right).		
\end{equation}
Clearly, \eqref{3.fourier} gives 
\begin{equation}
\begin{aligned}
Z^{n+1}(j)& = M(j) Z^n(j) + F^{n+1}(j)\\
& = (M(j))^{n-k+2}Z^{k-1}(j)+\sum_{\ell=k}^{n+1}(M(j))^{n+1-\ell}F^\ell(j),\quad \forall n \ge k.
\end{aligned}
\end{equation}
Now for each fixed $j$, by Lemma \ref{le3.2}, we have
\begin{equation}
\label{3.1-est}
|(M(j))^{n-k+2}Z^{k-1}(j)| \le |Z^{k-1}(j)|,\quad
|(M(j))^{n+1-\ell}F^\ell(j)| \le K_1\rho_1^{n+1-\ell}|F^\ell(j)|,
\end{equation}
where $K_1>0$ depends only on $\tau_0,$ and $0<\rho_1<1$ depends only on $\tau_0$. From \eqref{3.1-est} we obtain that
\begin{equation}
\sup_{n \ge k}\sup_{j\in\mathbb Z^2\setminus\{(0,0)\}}
|Z^{n+1}(j)| \le C_1,
\end{equation}
where $C_1$ depends only on $(u^0,\cdots,u^{k-1},\tau_0,\gamma_0)$. Using \eqref{3.fourier} we get
\begin{equation}
\left||j|^4(\wi{u^{n+1}}(j))^2\right| \le\frac{C_2}{|j|^4}+|\wi{f^n}(j)|^2,\quad \forall ~n \ge k, 
\end{equation}
where $C_2$ depends only on $(u^0,\cdots,u^{k-1},\tau_0,\gamma_0)$. The desired $H^2$-bound then follows easily.
\end{proof}

\begin{lem}
\label{le3.1}
Let $2 \le k \le 5$ and $0<s_0<\frac{1}{A_0^{(k)}}$. For $0<s \le s_0$ the roots $\lambda_i(s),~i=1,\cdots,k$, to the equation in $\lambda$
\begin{equation}\label{eq:poly}
\lambda^k+A_1^{(k)}s\lambda^{k-1}+\cdots+A_k^{(k)}s=0	
\end{equation}
satisfy
\begin{equation}
\label{e-est}
\max_i|\lambda_i(s)| \le\lambda_a<1,
\end{equation}
where $\lambda_a>0$ depends only on $s_0.$
\end{lem}

\begin{rem}
We first verify this lemma numerically. Figure \ref{fig:max_lambda} plots the $\max_i|\lambda_i(s)|$ w.r.t. $s$ for different BDF$k$ with $2 \le k  \le 7$. 
It can be seen that this lemma holds true for BDF$k$,~$2\le k\le6$, but not for BDF7. In fact, it is well-known that BDF$k$ is unstable when $k \ge 7$.
\begin{figure}
\includegraphics[trim={2in 0.5in 2in 0},clip,width = 1\textwidth]{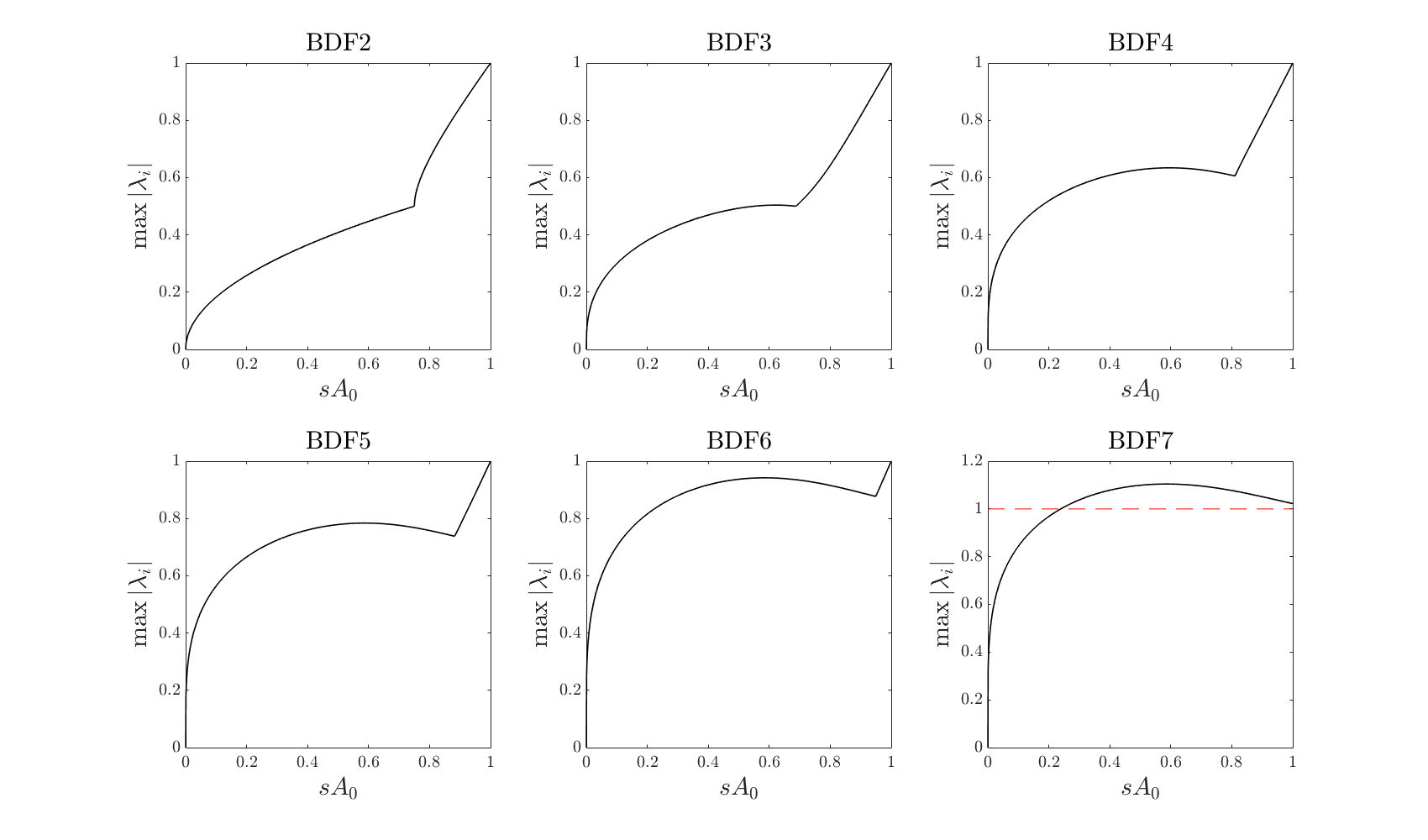}
\caption{$\max_i|\lambda_i(s)|$ w.r.t. $s$ for different BDF$k$, $2 \le k  \le 7$.}\label{fig:max_lambda}
\end{figure}
\end{rem}

\begin{proof}
In the BDF$2$ case, the proof is trivial and is omitted here. 
Next, we prove the cases of $k=3,4$ based on the discriminant of algebraic equation, while for the case of $k=5$, we prove it based on constructing a meromorphic function associated with the polynomial. 
\medskip

\noindent{\bf BDF$3$ case.} The algebraic equation \eqref{eq:poly} reads as
\begin{equation}
\label{bdf-3}
F_3(\lambda,s)=\lambda^3-3 s\lambda^2+\frac32s\lambda-\frac13s=0,\quad s\in\left(0,\frac{6}{11}\right).
\end{equation}
It is known that the discriminant of a cubic polynomial equation
$ax^3+bx^2+cx+d=0$
is
\begin{equation*}
\triangle_3=b^2c^2-4ac^3-4b^3d-27a^2d^2+18abcd.
\end{equation*}
If $\triangle_3<0$, then the cubic polynomial equation admits one real root and one pair of non-real complex conjugate roots. By a simple computation, we obtain the discriminant for \eqref{bdf-3} as
\begin{equation}
\label{3.dis}
\triangle_3 = -3s^2+\frac{27}{2}s^3-\frac{63}{4}s^4<0\quad \mbox{for}\quad s\in\left(0,\frac{6}{11}\right).
\end{equation}
Therefore, \eqref{bdf-3} admits only one real root and two complex conjugate roots. Since 
$$\partial_s F_3(\lambda,s) = -3\lambda^2+\frac32\lambda-\frac13<0\quad \mbox{for}\quad \lambda\in\mathbb{R},$$
we have
$$F_3(1,s)>F_3\left(1,\frac{6}{11}\right) = 0 \quad \mbox{for}\quad s\in\left(0,\frac{6}{11}\right).
$$
On the other hand, 
$$F_3\left(\frac12s,s\right)= -s\left(\frac{5}{8}s^2-\frac{3}{4}
s+\frac13\right)<0\quad \mbox{for}\quad s\in\left(0,\frac{6}{11}\right).$$
Then, the only real root of \eqref{bdf-3} satisfies
\begin{equation}
\lambda_{1}(s)\in\left(\frac12s,1\right)\quad \mbox{for}\quad s\in\left(0,\frac{6}{11}\right).
\end{equation}
Since $\lambda_1(s)\lambda_2(s)\lambda_3(s) = -\frac13 s$ and $\lambda_2(s) = \overline{\lambda_3(s)}$, we have
\begin{equation}
|\lambda_{2}(s)|=|\lambda_3(s)|<\sqrt{\frac23}<1\quad \mbox{for}\quad s\in\left(0,\frac{6}{11}\right).
\end{equation}
Hence, we proved the conclusion \eqref{e-est} for BDF$3.$
\medskip

\noindent {\bf BDF$4$ case.} The characteristic equation reads 
\begin{equation}
\label{bdf-4}
F_4(\lambda,s)=\lambda^4-4s\lambda^3+3s\lambda^2-\frac43s\lambda+\frac14s=0,~s\in\left(0,\frac{12}{25}\right).
\end{equation}
The discriminant for quartic polynomial equation
$$ax^4+bx^3+cx^2+dx+e=0$$
is
\begin{equation*}
\begin{aligned}
\triangle_4=~&256a^3e^3-192a^2bde^2-128a^2c^2e^2+144a^2cd^2e-27a^2d^4\\
&+144ab^2ce^2-6ab^2d^2e
-80abc^2de+18abcd^3+16ac^4e-4ac^3d^2\\
&-27b^4e^2+18b^3cde-4b^3d^3-4b^2c^3e+b^2c^2d^2.
\end{aligned}
\end{equation*}
If $\triangle_4<0$, then quartic polynomial equation has two distinct real roots and two complex conjugate non-real roots. While if $\triangle_4>0$, $P=8ac-3b^2>0$, the quartic equation has two pairs of non-real complex conjugate roots. By tedious computation, we see that the discriminant for \eqref{bdf-4} is
\begin{equation}
\label{b4.dis}
\triangle_4=4s^3-\frac{88}{3}s^4+\frac{220}{3}s^5-\frac{1696}{27}s^6
\begin{cases}
>0\quad \mbox{for}\quad s\in\left(0,\frac38\right),\\
\\
<0\quad  \mbox{for}\quad s\in\left(\frac38,\frac{12}{25}\right),
\end{cases}
\end{equation}
and $P=24 s- 48 s^2>0$ for $x\in\left(0,\frac38\right).$ Hence, equation \eqref{bdf-4} admits two real roots and one pair of complex conjugate roots when $s\in\left(\frac38,\frac{12}{25}\right)$ and two pairs of complex conjugate roots when $s\in\left(0,\frac38\right)$. Note that for $s=\frac38$, the equation \eqref{bdf-4} possesses repeated real roots and a pair of complex conjugate roots.

\noindent Case 1. $s\in\left[\frac38,\frac{12}{25}\right)$. When $s=\frac{12}{25}$, equation \eqref{bdf-4} admits two real roots $1$ and $\lambda_*\approx 0.3814$. We notice that
$$\partial_sF_4(\lambda,s)=-4\lambda^3+3\lambda^2-\frac43\lambda+\frac14<0\quad \mbox{for}\quad \lambda\in\left(0.32,1\right).$$
It follows that
$$F_4(1,s)>F_4\left(1,\frac{12}{25}\right)\quad \mbox{for}\quad s\in\left(0,\frac{12}{25}\right),$$
and
$$F_4(\lambda_*,s)>F_4\left(\lambda_*,\frac{12}{25}\right)\quad \mbox{for}\quad s\in\left(0,\frac{12}{25}\right).$$
This implies for any $s\in\left[\frac38,\frac{12}{25}\right)$, the two real roots $\lambda_1(s),\lambda_2(s)$ are locked in $(\lambda_*,1)$ (see Figure \ref{fig:F} for a schematic diagram).  On the other hand the pair of complex conjugate roots satisfy
\begin{equation*}
|\lambda_3(s)|=|\lambda_4(s)|<\sqrt{\frac{s}{4}\cdot\frac{1}{\lambda_*^2}}\le 
\sqrt{\frac{3}{25}\cdot\frac{1}{0.3813^2}}
<0.92\quad \mathrm{for}\quad s\in\left[\frac38,\frac{12}{25}\right).
\end{equation*}
Hence, in this case, we proved that $\max_i|\lambda_i|<1.$
\begin{figure}
\includegraphics[trim={0in 0.0in 0in 0},clip,width = 0.54\textwidth]{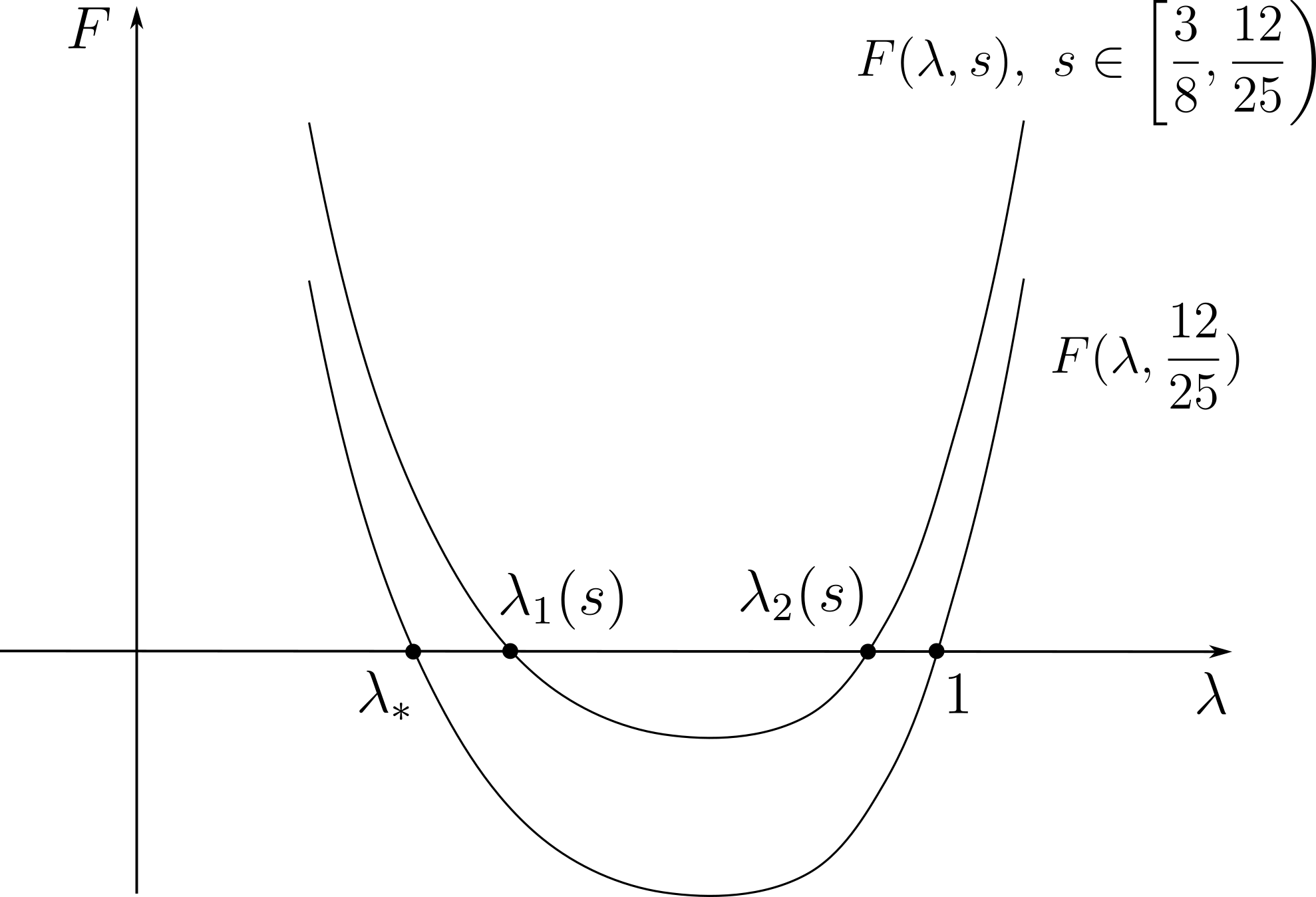}
\caption{Schematic diagram of $F(\lambda,s)$ in BDF$4$ case.}\label{fig:F}
\end{figure}

\noindent Case 2. $s\in\left(0,\frac38\right)$. In this case the equation \eqref{bdf-4} admits two pairs of complex conjugate roots. Denoted them by $c_1, \bar{c}_1, c_2, \bar{c}_2$. Clearly, 
\begin{equation}
\label{sys}
\begin{cases}
\Re(c_1)+\Re(c_2)=2s,\\ |c_1|^2+|c_2|^2+4\Re(c_1)\Re(c_2)=3s,\\
|c_1|^2\Re(c_2)+|c_2|^2\Re(c_1)=\frac23s,\\
|c_1|^2|c_2|^2=\frac14s.
\end{cases}
\end{equation}
 Without loss of generality, we may assume from the last equation in \eqref{sys} that
\begin{equation}
\label{b4.assume}
|c_1|^2 \ge1>\frac14s \ge|c_2|^2.
\end{equation}
Together with the first and third equation in \eqref{sys}, we get
\begin{equation}\label{3.26}
\Re(c_1) = \frac{s(2|c_1|^2-\frac23)}{|c_1|^2-|c_2|^2}>0,\quad 
\Re(c_2) = \frac{s(-2|c_2|^2+\frac23)}{|c_1|^2-|c_2|^2}>0.
\end{equation}
Consequently,
\begin{equation}\label{3.27}
|c_1|^2-|c_2|^2=3s-4\Re(c_1)\Re(c_2)-2|c_2|^2<3s.
\end{equation}
Substituting \eqref{3.26} into the second equation in \eqref{sys}, we have
\begin{equation}
\begin{aligned}
3s& =|c_1|^2 + |c_2|^2+  \frac{4s^2(2|c_1|^2-\frac23)(-2|c_2|^2+\frac23)}{(|c_1|^2-|c_2|^2)^2} \\
& \stackrel{\text{by}~\eqref{3.27}}{\geq} |c_1|^2+|c_2|^2+\frac{4s^2(\frac43(|c_1|^2+|c_2|^2)-4|c_1|^2|c_2|^2-\frac49)}{(3s)^2}\\
& \stackrel{\text{by}~ |c_1|^2|c_2|^2 = \frac14 s}{=}|c_1|^2+|c_2|^2+\frac49\left(\frac43(|c_1|^2+|c_2|^2)-s-\frac49\right)\\
&\stackrel{\text{by}~|c_1|^2+|c_2|^2 \ge 1 }\geq 1+\frac{16}{27}-\frac{16}{81}-\frac49s = \frac{113}{81}-\frac49s,
\end{aligned}
\end{equation}
which implies that $s>\frac{113}{279}>\frac{3}{8}$. Thus we arrive at a contradiction. Therefore, we finish the proof for this case.

\medskip

\noindent {\bf BDF$5$ case.} The algebraic equation reads 
\begin{equation}
\label{bdf-5}
F_5(\lambda,s)=\lambda^5-5s\lambda^4+5s\lambda^3-\frac{10}{3}s\lambda^2+\frac54s\lambda-\frac15s,\quad s\in\left(0,\frac{60}{137}\right).
\end{equation}
For any $s\in\left(0,\frac{60}{137}\right),$ we consider the following meromorphic function $\mathcal{F}_s(z)$ defined in complex domain 
\begin{equation}
\label{bdf-5-1}
\mathcal{F}_s(z)=\dfrac{F_5(z,s)}{z^5-\frac{137}{60}sz^4}.
\end{equation}
It is easy to see that $z^5-\frac{137}{60}sz^4$ never vanishes in $|z| \ge1$ whenever $s\in\left(0,\frac{60}{137}\right).$ Therefore $\mathcal{F}_s(z)$ defines a holomorphic function for $|z| \ge1.$ 
It is easy to see that
\begin{equation}\label{eq:limFs}
\lim_{|z|\to\infty}\mathcal{F}_s(z)=1>0.
\end{equation}
Next, we prove that
\begin{equation}
\label{bdf-5-claim}
\Re(\mathcal{F}_s(z))>0\quad \mbox{for}\quad |z| \ge1~\mbox{and}~s\in\left(0,\frac{60}{137}\right),
\end{equation}
implying that $\mathcal{F}_s(z)$ will not vanish in $|z| \ge 1$.

Note that $\Re(\mathcal{F}_s(z))$ is harmonic when $|z|>1$. 
According to the maximum principle of harmonic function and \eqref{eq:limFs}, to prove \eqref{bdf-5-claim}, it is sufficient to show that 
\begin{equation}
\Re(\mathcal{F}_s(z))>0\quad \mbox{for}\quad |z|=1~\mbox{and}~s\in\left(0,\frac{60}{137}\right),
\end{equation}
which is equivalent to
\begin{equation}
\label{bdf-5-claim-1}
\Re\left(\left(1-5s\bar z+5s\bar z^2-\frac{10}{3}s\bar z^3+\frac54s\bar z^4-\frac15s\bar z^5\right)\left(1-\frac{137}{60}s z\right) \right)>0,
\end{equation}
for any $|z|=1$ and $s\in\left(0,\frac{60}{137}\right)$, where $\bar z$ denotes the conjugate of $z$. 

We write $\Re(z)$ as $x$ and $|x| \le1$, then by the trigonometric identities and tedious computations, we have
\begin{equation}
\label{bdf-5-p}
\begin{aligned}
\Re\left(\left(1-5s\bar z+5s\bar z^2-\frac{10}{3}s\bar z^3+\frac54s\bar z^4-\frac15s\bar z^5\right)\left(1-\frac{137}{60}s z\right)\right)=1+A(x)s+B(x)s^2,
\end{aligned}	
\end{equation}
where
\begin{align*}
A(x)=~&-\frac{15}{4} + \frac{103}{60} x - \frac{28}{3}x^3 + 10 x^4 - \frac{16}{5} x^5,\\
B(x)=~&\frac{959}{225} -\frac{137}{48} x + \frac{2603}{225} x^2 - \frac{137}{12} x^3 + \frac{274}{75} x^4.
\end{align*}
Regarding the right-hand side of \eqref{bdf-5-p} as a quadratic polynomial in $s$, we derive the discriminant 
\begin{equation}
\label{bdf-5-d}
\begin{aligned}
A^2(x)-4B(x)=&-\frac{10751}{3600} - \frac{35}{24} x - \frac{155983}{3600} x^2 + \frac{347}{3} x^3 - \frac{27373}{225}x^4
+\frac{175}{3}x^5\\
&+ \frac{17128}{225} x^6 - \frac{560 }{3}x^7 +\frac{2396}{15} x^8- 64 x^9 + \frac{256}{25} x^{10}.
\end{aligned}
\end{equation}
In Figure \ref{fig:AB} given by Matlab, we can see that the polynomial defined on the right-hand side of \eqref{bdf-5-d} admits two real roots in $\mathbb{R}$, they are $x_s\approx-0.908$ and $1$. 
The discriminant is negative when $ x_s<x< 1$. 
On the other hand, it can be shown that $B(x)>0$ for $-1\le x\le1$ and $A(x)>0$ for $-1 \le x<x_s$. As a consequence, $1+A(x)s+B(x)s^2>0$ for $-1 \le x<x_s,~s\in\left(0,\frac{60}{137}\right).$ 
While if $x=1$, the quadratic polynomial is
\begin{equation*}
1+A(x)s+B(x)s^2=\left(1-\frac{137}{60}s\right)^2>0\quad \mbox{for}\quad s\in\left(0,\frac{60}{137}\right).
\end{equation*}
Hence, we have shown that \eqref{bdf-5-claim-1} holds for all $|x| \le 1$ and the claim \eqref{bdf-5-claim} then holds. 
Therefore, $\mathcal{F}_s(z)$ never vanishes in $|z| \ge1$ for $s\in\left(0,\frac{60}{137}\right).$ 
It follows that
\begin{equation}
F_5(s,z)~\mbox{admits no roots in}~|z| \ge1\quad\mbox{for}\quad s\in\left(0,\frac{60}{137}\right).
\end{equation}
It is not difficult to check   that all roots of $F_5(s,z)=0$ are in the unit ball if $s$ is close to $0$. Thus, we proved the conclusion \eqref{e-est} for BDF5 case.

\begin{figure}
\includegraphics[trim={0in 0.in 0in 0},clip,width = 0.55\textwidth]{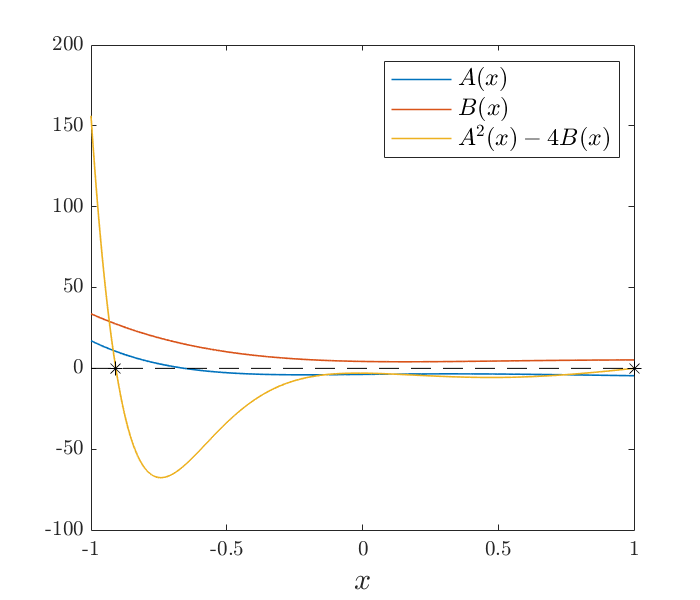}
\caption{$A(x)$, $B(x)$, and $A^2(x)-4B(x)$ w.r.t. $x$ for BDF$5$ case.}\label{fig:AB}
\end{figure}

\end{proof}

\begin{rem}
It is also possible to work out a proof for BDF$4$ by following
the meromorphic approach in the BDF$5$ case.
\end{rem}

\begin{lem}
\label{le3.2}	
Let $0<s_0<\frac{1}{A_0^{(k)}}$. Consider the matrix
\begin{equation}
M(s)=\left(
\begin{matrix}
-A_1^{(k)}s  &  -A_2^{(k)}s &\cdots &-A_{k-1}^{(k)}s&-A_{k}^{(k)}s\\
1  & 0 &\cdots &0&0\\
0  & 1 &\cdots &0&0\\
\vdots &\vdots &\ddots &\vdots&\vdots\\
0&0&\cdots&1&0
\end{matrix}
\right),	
\end{equation}
where $0<s \le s_0.$ There exists an integer $n_0 \ge1$ which depends on $s_0$ such that
\begin{equation}\label{ineq:eps0}
\sup_{0<s \le s_0}\sup_{x\in\mathbb{R}^k,|x|=1}
|M(s)^{n_0}x| \le\epsilon_0<1,
\end{equation}
where $\epsilon_0>0$ depends only on $s_0$ and $|x|=\sqrt{\sum_{i=1}^k|x_i|^2}$ denotes the usual $\ell^2$-norm on $\mathbb{R}^k.$
\end{lem}

\begin{proof}
First we notice that
\begin{equation}
M(s)^k=sM_k(s),
\end{equation}
where all entries of $M_k(s)$ are either constants or polynomials of $s$. Therefore, if $s_1$ is sufficiently small, then we have 
\begin{equation}
\sup_{0<s \le s_1}\sup_{x\in\mathbb{R}^k,|x|=1}|M(s)^kx| \le\frac12.
\end{equation}
We now focus on the regime $s_1 \le s \le s_0<\frac{1}{A_0}$. Consider a fixed $s_*\in[s_1,s_0]$. By the above lemma, there exists $n_*$ depending on $s_*$ such that
\begin{equation*}
\sup_{x\in\mathbb R^k,~|x|=1}|M(s_*)^{n_*}x| \le\epsilon_*<1,
\end{equation*}
where $\epsilon_*$ also depends on $s_*$. Perturbing around $s_*$ we can find a small neighborhood $J_*$ around $s_*$ such that 
\begin{equation*}
\sup_{x\in\mathbb R^k,~|x|=1}|M(s)^{n_*}x| \le\epsilon_1<1,
\quad \forall s\in J_*,
\end{equation*}
where $\epsilon_1$ depends only on $s_*.$
By a covering argument, we conclude that there exist $n_0$ and $\epsilon_0$ such that \eqref{ineq:eps0} is satisfied. 
\end{proof}

\begin{rem}
The convergence analysis of the BDF$k$/EP$k$ scheme can be done similarly as in  the BDF$3$ case 
(cf. \cite{li2021bdf3}). 
\end{rem}

%

 \vspace{1cm}
{\bf Acknowledgement.}
The research of W. Yang is supported by NSFC Grants 11801550, 11871470, and 12171456.
The work of C. Quan is supported by NSFC Grant 11901281, the Guangdong Basic and Applied Basic Research Foundation (2020A1515010336), and the Stable Support Plan Program of Shenzhen Natural Science Fund (Program Contract No. 20200925160747003).

\end{document}